\documentclass[12pt]{amsart}
\usepackage{amsmath,amssymb,amsthm}
\usepackage{setspace}
\usepackage{moreverb}
\usepackage[dvips]{graphicx}

\newcommand{\ignore}[1]{}

\theoremstyle{definition}

\theoremstyle{remark}

\numberwithin{equation}{section}

\input xy
\xyoption{all}

\newcommand{\Hom}{\mbox{\rm{Hom}} }

\newcommand{\fM}{{\mathfrak m}}

\newcommand{\cC}{{\mathcal C}}

\newcommand{\lra}{{\longrightarrow}}

\theoremstyle{plain}
  \newtheorem{thm}{Theorem}
  
  \newtheorem{lem}[thm]{Lemma}
  \newtheorem{cor}[thm]{Corollary}
  
\theoremstyle{definition}
  
  \newtheorem{exmp}[thm]{Example}

  \newtheorem{rem}[thm]{Remark}

\let\opn\operatorname 
\newcommand\chara{\opn{char}}
\newcommand\supp{\opn{supp}}
\newcommand\supptw{\supp_{\mathbf 2}}
\newcommand\cost{\opn{cost}}
\def\m{\mathsf{m}}
\def\n{\mathsf{n}}
\def\cC{\mathcal {C}}

\def\bx{\mathbf {x}}
\def\one{\mathbf {1}}

\def\RR{\mathbb R}
\def\NN{\mathbb N}

\begin{document}

\title[Addendum to "Frobenius and Cartier algebras of Stanley-Reisner rings"]{Addendum to "Frobenius and Cartier algebras of Stanley-Reisner rings" [J. Algebra  358 (2012) 162-177]}

\author[J. \`Alvarez Montaner]{Josep \`Alvarez Montaner}
\thanks{The first author is partially supported by SGR2009-1284 and MTM2010-20279-C02-01}
\address{Dept. Matem\`atica Aplicada I\\
Univ. Polit\`ecnica de Catalunya\\ Av. Diagonal 647,
Barcelona 08028, SPAIN} \email{Josep.Alvarez@upc.edu}

\author[K. Yanagawa]{Kohji Yanagawa}
\thanks{The second author is supported by JSPS KAKENHI Grant Number 22540057}
\address{Department of Mathematics, Kansai University, Suita 564-8680, Japan}
\email{yanagawa@ipcku.kansai-u.ac.jp}


\begin{abstract}
We give a purely combinatorial characterization of complete Stanley-Reisner rings having a principally (equivalently, finitely generated) generated Cartier algebra.
\end{abstract}


\maketitle

\section{Introduction}
Let $(R, \fM)$ be a complete local ring of prime characteristic $p>0$. The notion of Cartier algebra, introduced
by K.~Schwede \cite{Sch11} and developed by M.~Blickle \cite{Bli13}, has received a lot of attention
due to its role in the study of test ideals. More precisely, the ring of Cartier operators on $R$ is the
graded, associative, not necessarily commutative  ring
$$\cC(R):=\bigoplus_{e\geq 0}\Hom_R(F_{\ast}^eR,R),$$ where $F^e_{\ast}R$ denotes the ring $R$ with the
left $R$-module structure given by the $e$-th iterated Frobenius map $F^e:R\lra R$, i.e.  the left
$R$-module structure given by $r\cdot m:=r^{p^e}m$.
One should mention that, using Matlis duality, the Cartier algebra of $R$ corresponds to the Frobenius algebra of the
injective hull of the residue field $E_R(R/\fM)$ introduced by G.~Lyubeznik and K.~E.~Smith in \cite{LS01}.

\vskip 2mm

Let $S=K[\![x_1, \ldots, x_n]\!]$ be the formal power series ring
over a field $K$. In this note we will assume that $\chara(K)=p>0$.
Given a simplicial complex $\Delta$ with vertex set $[n]:= \{1,2, \ldots, n\}$ one may associate
a squarefree monomial ideal $I_\Delta:=( \prod_{i \in F} x_i \mid F \subset [n], F \not \in \Delta)$ in $S$
via the  {\it Stanley-Reisner correspondence}.
Building upon an example of M.~Katzman \cite{Kat10}, the first author together with A.~F.~Boix and
S.~Zarzuela \cite{ABZ12} studied Cartier algebras of 
{\it complete Stanley-Reisner rings $R:=S/I_\Delta$} associated to $\Delta$.
One of the main results
obtained in \cite{ABZ12} is that these Cartier algebras can be either principally generated or infinitely generated as an $R$-algebra.

\begin{thm}[{\cite[Theorem~3.5]{ABZ12}}]\label{ABZ main}
With the above notation,  set $R:=S/I_\Delta$.
Assume that each $x_i$ divides some minimal monomial generator of $I_\Delta$. 
Then, the following are equivalent:

\begin{itemize}
 \item [(1)] The Cartier algebra $\cC(R)$ is principally generated.

 \item [(2)] $I_{\Delta}^{[2]}:I_\Delta=I_\Delta^{[2]}+ (\bx^\one)$.

\end{itemize}

\noindent Otherwise the Cartier algebra $\cC(R)$ is infinitely generated.
Here $I_{\Delta}^{[2]}$ denotes the second Frobenius power of $I_{\Delta}$ and $\bx^\one:= x_1x_2 \cdots x_n$.

\end{thm}

\begin{rem} \label{remark}
Set $V:=\{ \, i \mid \text{$x_i$ divides some minimal monomial generator of $I_\Delta$} \, \}$.

\begin{itemize}
 \item [(i)]
The condition that {\it each $x_i$ divides some minimal monomial generator of $I_\Delta$} (equivalently, $V=[n]$) is only used to simplify
the notations of Theorem \ref{ABZ main}. If it is not satisfied (equivalently,  $\Delta$ is a cone over some vertex), 
we have that $\cC(R)$ is principally generated if and
only if $I_{\Delta}^{[2]}:I_\Delta=I_\Delta^{[2]}+ (\prod_{i \in V}x_i)$.
We can always reduce to the case $V=[n]$ since there is always a simplicial complex $\Delta'$ on $V$ such that
$\Delta$  coincides with the simplicial join $\Delta' * 2^{[n] \setminus V} := 
\{ \, F \cup G \mid F \in \Delta',  \, G \subset [n] \setminus V \, \}$,  
so the result follows from Lemma \ref{join} below.

\vskip 2mm

\item [(ii)] The original result in \cite{ABZ12} has a slightly different formulation in terms of the colon ideals
$I_{\Delta}^{[p^e]}:I_\Delta$, $e\geq 1$,  but it was already noticed in \cite[Remark~3.1.2]{ABZ12} that one may reduce to the case
$p=2$ and $e=1$. We also point out that Theorem \ref{ABZ main}
also holds in the case $\opn{ht}(I_\Delta)=1$ that was treated separately in \cite{ABZ12} for clearness.
\end{itemize}

\end{rem}

\begin{lem} \label{join}
Let  $\Delta$ be a simplicial complex with vertex set $[n]$. Assume that there exists
a simplicial complex $\Delta'$ on $V \subseteq [n]$ such that $\Delta = \Delta' * 2^{[n] \setminus V}$.
Then, $\cC(S/I_\Delta)$ is principally generated if and only if so is $\cC(S'/I_{\Delta'})$, where $S':=K[\![x_i \mid i \in V]\!]$.
\end{lem}

\begin{proof}
We have $S/I_\Delta \cong (S' /I_{\Delta'}) 
[\![x_i \mid i \not \in V]\!]$. Then the result follows from the description
of the Cartier algebra in terms of the colon ideals
$I_{\Delta}^{[2]}:I_\Delta$ (see \cite{ABZ12} and the references therein).
\end{proof}

\section{A characterization of principally generated Cartier algebras}
The Cartier algebra of an $F$-finite complete Gorenstein local ring $R$ is principally generated
as a consequence\footnote{Using Matlis duality.} of \cite[Example 3.7]{LS01}.
The converse holds true for $F$-finite normal rings (see \cite{Bli13}). Complete Stanley-Reisner rings
are $F$-finite but, almost always non-normal and when discussing examples at the boundary of the Gorenstein property
one can even find examples of principally generated Cartier algebras that are not even Cohen-Macaulay.
The authors of \cite{ABZ12} could not find the homological conditions that tackle this property so
the aim of this note is to address this issue. Our main result is a very simple  combinatorial criterion
in terms of the simplicial complex $\Delta$. To this purpose we recall
that a  {\it facet} of $\Delta$ is a maximal face with respect to inclusion.
We say $F \in \Delta$ is a {\it free face} if $F \cup \{ i\}$ is a facet  for some $i \not \in F$ 
and $F \cup \{ i\}$ is  the unique facet containing $F$. 
With the above situation, the process of passing from $\Delta$ to $\Delta \setminus \{ F, F \cup \{ i \}  \}$ preserves 
the homotopy type,  and is called an {\it elementary collapse}.  See, for example, \cite[Definition~3.3.1]{M80}. 

\begin{thm}\label{main this time}
Under the same assumptions as in Theorem~\ref{ABZ main}, the following are equivalent.
\begin{itemize}
\item[(a)] The Cartier algebra $\cC(R)$ is principally generated.

\item[(b)] $\Delta$ does not have a free face, that is, we can not apply an elementary collapse to $\Delta$.
\end{itemize}
\end{thm}

\begin{proof}
For a monomial $\m =\prod_{i=1}^n x_i^{a_i} \in S$, set
$\supp (\m) : =\{ \, i \mid a_i \ne 0 \, \}$ and
$\supptw (\m) :=\{ \, i \mid a_i \ge 2 \, \}$.
Note that  $\m \in I_\Delta^{[2]}$ if and only if $\supptw(\m) \not \in \Delta$.
Further more,  under the assumption that $\supp(\m) \ne [n]$,
we have $\m \in I_\Delta^{[2]}+(\bx^\one)$ if and only if $\supptw(\m) \not \in \Delta$.

(a) $\Rightarrow$ (b): 
By Theorem~\ref{ABZ main}, it suffices to show that  $I_{\Delta}^{[2]}:I_\Delta=I_\Delta^{[2]}+ (\bx^\one)$  implies (b), and the same is true for the proof of 
the converse implication.   

Assume that $\Delta$ does not satisfy (b).
Then we may assume that $\{ 1, 2, \ldots, l \}$ is a free face, and
it is contained in a unique facet  $\{ 1, 2, \ldots, l+1\}$. Set
$$\m:= \left( \prod_{i=1}^l x_i^2  \right) \cdot \left( \prod_{i=l+2}^n x_i \right).$$
Clearly,  $\m \not \in I_\Delta^{[2]}+(\bx^\one)$.
Take any monomial $\n \in I_\Delta$. Since  $\{ 1, 2, \ldots, l+1\} \in \Delta$,
$\n$ can be divided by $x_j$ for some $l+2 \le j \le n$. Then
$\supptw(\m\n) \supseteq \{1,2, \ldots, l, j\}$, which is not a face of $\Delta$.
It follows that  $\m\n \in I_\Delta^{[2]}$. Summing up, we have
$\m \in (I_\Delta^{[2]}:I_\Delta) \setminus I_\Delta^{[2]}  +(\bx^\one)$. Hence the condition (a) does not hold, and we are done.

(b) $\Rightarrow$ (a): Assume that the condition  (b) is satisfied.
Since   $I_{\Delta}^{[2]}:I_\Delta \supseteq I_\Delta^{[2]}+ (\bx^\one)$ always holds, it suffices to
prove that $I_\Delta^{[2]}:I_\Delta\subseteq  I_\Delta^{[2]}+ (\bx^\one)$, equivalently, $\m \not \in  I_\Delta^{[2]}+ (\bx^\one)$
implies $\m \not \in  I_\Delta^{[2]}: I_\Delta$. So take a monomial $\m \in S$ with $\m \not \in I_\Delta^{[2]}+ (\bx^\one)$.
If $\# \supp (\m) \le n-2$ and $i \not \in \supp (\m)$, then $x_i \m \not \in  I_\Delta^{[2]}+ (\bx^\one)$,  and
$x_i \m \not \in  I_\Delta^{[2]}:I_\Delta$ implies  $\m \not \in  I_\Delta^{[2]}:I_\Delta$. Hence we can replace $\m$ by $x_i \m$
in this case.
Repeating this operation, we may assume that $\# \supp(\m)=n-1$.
Let $x_l$ be the only variable which does not divide $\m$.

Set $F:= \supptw (\m)$.
Since $\m \not \in I_\Delta^{[2]}$, we have $F \in \Delta$. Moreover, there is a facet $G \in \Delta$
with $G \supseteq F$ and $l \not \in G$. To see this, take any facet $H \in \Delta$ with $H \supseteq F$.
If $l \not \in H$, then we can take $H$ as $G$.  If $l \in H$, then $H \setminus \{l\}$ is contained in
a facet $H'$ other than $H$ by the condition (b). Clearly,  we can take $H'$ as $G$.
Replacing $\m$ by $(\prod_{i \in G \setminus F} x_i) \cdot \m$, we may assume that $F=\supptw(\m)$ is a facet with  $l \not \in F$. Then $\n := x_l \cdot  \prod_{i \in F} x_i$
is contained in $I_\Delta$, since $\supp (\n)= F \cup \{ l\}$ is not a face of $\Delta$.
However,  we have $\supptw(\m \n) =F \in \Delta$ and $\m \n \not \in I_\Delta^{[2]}$.
It follows that $\m \not \in I_\Delta^{[2]}: I_\Delta$. This is what we wanted to prove.
\end{proof}

\begin{rem} \label{rem}
Under the assumption that each variable $x_i$ divides some minimal monomial generator of $I_\Delta$, equivalently
$\Delta$ is not a cone over any vertex, one may check out
that the condition on the Cartier algebra of a complete Stanley-Reisner ring $R=S/I_\Delta$ being principally generated is
a topological property of the geometric realization $X$ of $\Delta$. 
In fact, by Theorem~\ref{main this time}, $\cC(R)$ is 
not principally generated if and only if there is an open subset $U \subset X$ which is homeomorphic to $\{ \, (x_1, \ldots, x_m ) \in \RR^m \mid x_m \ge 0\, \}$ for some $m \in \NN$. However, the condition that $\Delta$ is not a cone over any vertex is {\it not}  topological. 
In this sense, being  principally  generated is not  a topological condition. 
This is quite parallel to the relation between Gorenstein and Gorenstein* properties of simplicial complexes where
we have that $\Delta$ is Gorenstein if and only if
$\Delta = \Delta' * 2^{[n]\setminus V}$ for some Gorenstein* complex $\Delta'$ on
some $V \subseteq [n]$ and Gorenstein* is a topological property
(See\footnote{The notation  
$\Delta= \opn{core} \Delta$ in  \cite[\S II.5]{Sta96} corresponds to $V=[n]$ in our notation.} \cite[\S II.5]{Sta96}). 
\end{rem}

Despite the fact that using Theorem~\ref{main this time} one may construct many simplicial complexes
satisfying that the Cartier algebra $\cC(R)$ is principally generated, e.g. triangulations of manifolds without boundary, it seems that 
there is no tight relation to any homological conditions on $R$. 
The best we can say in this direction is the following. 
For the definitions of {\it Buchsbaum* complexes} and undefined terminologies
we refer to \cite{Sta96} and \cite{AW12}.
Here we just recall the notion of {\it contrastar} of a simplicial complex $\Delta$. For a face $F \in \Delta$, set 
 $\cost_\Delta(F) := \{ G \in \Delta : G \not \supset F \}$. Clearly, this is a subcomplex of $\Delta$. 

\begin{cor}\label{Bbm*}
If $\Delta$ is Buchsbaum* (in particular, doubly Cohen-Macaulay, or Gorenstein*) over some field $K$, then $\cC(R)$ is
principally generated.
\end{cor}

\begin{proof}
Suppose that $\Delta$ is Buchsbaum* but $\cC(R)$ is not principally generated.  Since  $\Delta$ is Buchsbaum*,
$\Delta$ is not a cone over any vertex. Hence there is a free face $F$ contained in a unique facet $G$ by
Theorem~\ref{main this time}.
Clearly, $\cost_\Delta(F)=\Delta \setminus \{  F, G  \}$ and  $\cost_\Delta(G)=\Delta \setminus \{ G  \}$.
Hence, for $d=\dim \Delta$, we have $H_d(\Delta, \cost_\Delta(F) ; K) =0$ and
$H_d(\Delta, \cost_\Delta(G) ; K) =K$, and the map
$$H_d(\Delta, \cost_\Delta(F) ; K)  \longrightarrow H_d(\Delta, \cost_\Delta(G) ; K)$$ can not be surjective.
It means that $\Delta$ is not Buchsbaum* so we get a contradiction.
\end{proof}

This result together with Lemma \ref{join} allows us to give a direct proof of the fact that a complete Gorenstein Stanley-Reisner ring
$S/I_\Delta$ has a principally generated Cartier algebra since we have 
$\Delta = \Delta' * 2^{[n]\setminus V}$ for some Gorenstein* complex $\Delta'$ on
some $V \subseteq [n]$.


\begin{exmp}

\begin{itemize}
 \item [(i)] Consider the $1$-dimensional simplicial complex $\Delta$ in Figure 1 below. $\Delta$ is Cohen-Macaulay and $\cC(R)$
is principally generated, but $\Delta$ is not doubly Cohen-Macaulay so it is not Buchsbaum* as well.

\item [(ii)] Let $\Delta$ be the simplicial complex with facets $\{ 1,2, 3 \}$, $\{ 1,2, 4 \}$,
$\{ 1,3, 4 \}$, $\{ 2, 3,4 \}$, $\{1,5\}$ and $\{2,5\}$ (see Figure~2 below). Then, $\cC(R)$ is principally
generated but $\Delta$ is not pure.
\end{itemize}
\begin{figure}[htbp]
\begin{minipage}{.48\textwidth}
\begin{center}
\includegraphics[height=4cm, width=4cm]{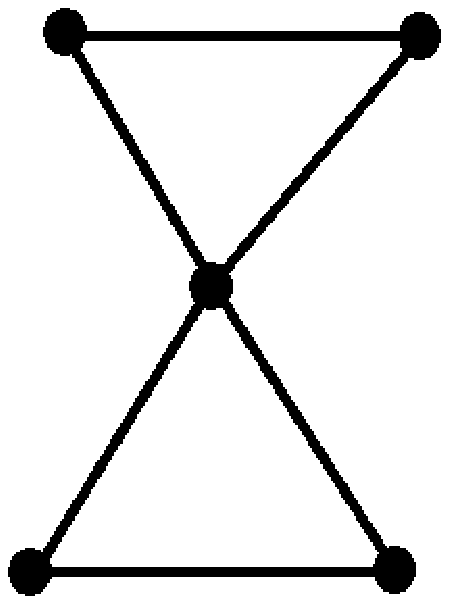}
\end{center}
\caption{}
\end{minipage}
\begin{minipage}{.48\textwidth}
\begin{center}
\includegraphics[height=4cm, width=4.3cm]{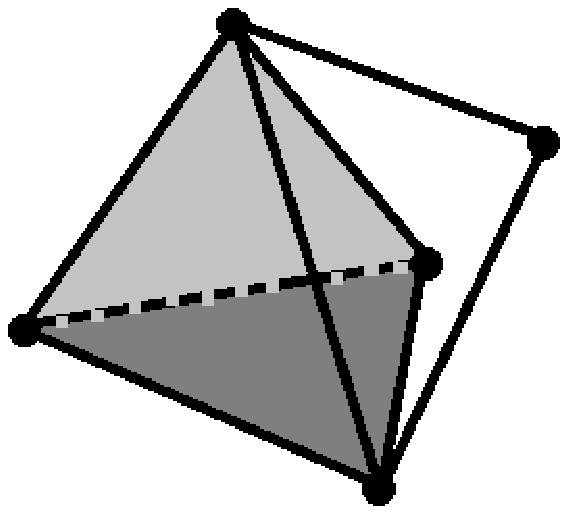}
\end{center}
\caption{}
\end{minipage}
\end{figure}
\end{exmp}

\section*{Acknowledgements}
We would like to thank A.~F.~Boix and S.~Zarzuela for many useful comments.
We also thank anonymous referee for valuable comments.

\end{document}